\documentclass[11pt,reqno]{amsart}
\usepackage{amsthm,amsfonts,amssymb,euscript,mathrsfs}
\usepackage[dvips]{graphicx}
\usepackage[margin=1in]{geometry}

\usepackage{hyperref}

\def\C{{\mathbb C}}
\def\Z{{\mathbb Z}}

\def\R{{\mathbb R}}

\def\e{{\varepsilon}}
\def\g{{\gamma}}

\def\a{{\alpha}}

\def\d{{\delta}}

\def\FF{{\mathcal{F}}}

\def\+R{+_{_{ \!\! \R}}}

\def\hat{\widehat}

\def\bar{\overline}

\def\S{Schr\"{o}dinger }

\DeclareMathAlphabet{\mathpzc}{OT1}{pzc}{m}{it}
\def\nl{\vglue0.5truemm\noindent}
\numberwithin{equation}{section}

\newtheorem{theorem}{Theorem}[section]
\newtheorem{lemma}[theorem]{Lemma}
\newtheorem{proposition}[theorem]{Proposition}

\newtheorem{remark}[theorem]{Remark}

\newtheorem{convention}{Convention}
\numberwithin{equation}{section}

\begin{document}

\title[Zakharov-type system]{On global solutions of a Zakharov type system}

\author{Thomas Beck}\address{Princeton University}\email{tdbeck@math.princeton.edu}

\author{Fabio Pusateri}\address{Princeton University}\email{fabiop@math.princeton.edu}

\author{Phil Sosoe}\address{Princeton University}\email{psosoe@math.princeton.edu}

\author{Percy Wong}\address{Princeton University}\email{pakwong@math.princeton.edu}

\thanks{The second author was partially supported by NSF grant DMS 1265875.}

\begin{abstract}
We consider a class of wave-Schr\"{o}dinger systems in three dimensions with a Zakharov-type coupling.
This class of systems is indexed by a parameter $\gamma$ which measures the strength of the null form in the nonlinearity of the wave equation.
The case $\g = 1$ corresponds to the well-known Zakharov system, while the case $\g=-1$ corresponds to the Yukawa system.
Here we show that sufficiently smooth and localized Cauchy data lead to pointwise decaying global solutions which scatter, for any $\g \in (0,1]$.
\end{abstract}


\maketitle


\section{Introduction}

\subsection{Statement of the problem and main result}
We will consider the following parametrized family of systems in three spatial dimensions:
\begin{equation}
\label{ZSg0}
\left\{
\begin{array}{l}
i \partial_t u + \Delta u  =  u n
\\
\\
\Box n = {\Lambda}^{1+\g} {|u|}^2
\end{array}
\right.
\end{equation}
where $\Box := - \partial_t^2 + \Delta$, $\Lambda := |\nabla| = \sqrt{-\Delta}$, and $-1 \leq \g \leq 1$. Here we have
\begin{align*}
u: (x,t) \in \R^3 \times \R \to \C, \quad n: (x,t) \in \R^3 \times \R \to \R \, . 
\end{align*}
The case $\g = 1$ corresponds to the well-known Zakharov system,
modeling propagation of Langmuir waves in an ionized plasma \cite{Zakharov}.
The case $\g = -1$ corresponds to the (massless) Yukawa system,
which is a model for the interaction between a meson and a nucleon.
\eqref{ZSg0} is a special case of the models introduced in \cite{SZ2} (see Section 3 there)
\begin{equation}
\label{ZS0}
\left\{
\begin{array}{l}
i \partial_t u + L_1 u =  u n
\\
\\
L_2 n = L_3 {|u|}^2 \, ,
\end{array}
\right.
\end{equation}
where $L_1,L_2$ and $L_3$ are constant coefficient differential operators.  This class of systems is referred to as Davey-Stewartson (DS) systems in the work of Zakahrov-Schulman \cite[Section 3]{SZ2}. In the recent literature the name DS is associated to a specific 2 dimensional system of the form \ref{ZS0}, modeling the evolution of weakly nonlinear water waves travelling predominantly in one direction, in which the wave amplitude is modulated slowly in two horizontal directions. See, for example, \cite{GhS1, GhS2}.

Here we will consider \eqref{ZSg0} in the range $0<\g\leq 1$ and prove the following

\begin{theorem}\label{maintheo}
Let $\g \in (0,1]$ be given.
Then there exist $N = N(\g) \gg 1$, and a small constant $\e_0 = \e_0(\g) > 0$, such that for any initial data 
$(u_0,n_0,n_1) = (u,n,\partial_t n)(t=0)$ satisfying\footnote{Here 
$\langle x \rangle$ is used to denote $\sqrt{1+|x|^2}$, and the Besov norm $\dot B^0_{1,1}$ is defined in \eqref{Besov}.}
\begin{align}
\label{data}
& {\| u_0 \|}_{H^{N+1}(\R^3)} + {\| {\langle x \rangle}^2 u_0 \|}_{L^2(\R^3)} \leq \e_0 \, ,
\\
&\label{data2} 
{\left\| \left( \Lambda n_0, n_1 \right) \right\|}_{H^{N-1}(\R^3)} 
+ {\left\| \langle \Lambda \rangle \left( \Lambda n_0,  n_1 \right) \right\|}_{\dot{B}^0_{1,1}(\R^3)}
+ {\left\| {\langle x \rangle}^2 \left( n_0,n_1 \right) \right\|}_{H^1(\R^3)} 
  \leq \e_0 \, ,
\end{align}
there exists a unique global-in-time solution $(u,n)(t)$ to the Cauchy problem associated to \eqref{ZSg0}.
Moreover, there exists $0 < \a < 1/6$\footnote{One can choose $\a \sim 1/6$ for $\g \geq 1/2$.} such that
\begin{align}
{\| u(t) \|}_{L^{\infty}(\R^3)} \lesssim \e_0 {(1+t)}^{-1-\a}  \quad , \quad {\| n(t) \|}_{L^\infty(\R^3)} \lesssim \e_0 {(1+t)}^{-1} \, .
\end{align}
As a consequence, the solution $(u,n)(t)$ approaches a linear solution as $t \rightarrow \infty$. 
\end{theorem}

The proof shows that the assumptions \eqref{data2} on the initial data are somewhat stronger than necessary.
We have chosen to display these conditions for simplicity. 
In Section 2, we will give the strategy of the proof, and describe some more of the properties satisfied by the solution $(u,n)(t)$.

\subsection{Motivation and previous results}

\subsubsection{The Zakharov system}
The Zakharov system, (\eqref{ZSg0} with $\g=1$),
\begin{equation}
\label{Z}
\tag{Z}
\left\{
\begin{array}{l}
i\partial_t u + \Delta u =  n u
\\
\\
\Box n = - \Delta |u|^2,
\end{array}
\right.
\end{equation}
was derived by V. Zakharov in \cite{Zakharov} to model Langmuir waves in plasma, and has since then
been under intensive investigation by physicists and mathematicians (see \cite{HamPlasma,SulemBook} for some background). 
As we remarked above, it is a particular example of Zakharov-Schulman systems, introduced in \cite{SZ2} and studied in \cite{HamPlasma}, 
\cite{KPV1}, \cite{KW}.

From the mathematical side, there has been considerable work on the local and global well-posedness of solutions 
with rough data  through the works of Kenig, Ponce and Vega \cite{KPV1}, 
Bourgain and Colliander \cite{BC}, Ginibre, Tsutsumi and Velo \cite{GTV}, Bejenaru, Herr, Holmer and Tataru \cite{BHHT} and Bejenaru and Herr \cite{BH}. 
In particular, global well-posedness for small data in the energy space was obtained in \cite{BC} by combining local well-posedness and conservation laws.
(See the references in the cited works for previous well-posedness results).
Many works have also dealt with singular limits related to the Zakharov system and with the rigorous derivation of the system in various limiting regimes from other equations and vice versa; see for example \cite{Texier}, 
\cite{MN} 
and references therein.

Concerning the scattering question, most of the previous work has been carried out for the  final value problem, i.e. data at $t = \infty$, 
as in the papers of Ozawa and Tsutsumi \cite{OT}, Ginibre and Velo \cite{GV-Zak}, and Shimomura \cite{ShimoZak}. Similar work has also been dedicated to other coupled systems of \S and wave equations, see for example the works of Ginibre and Velo \cite{GV-WS1,GV-WS2,GV-WS3}, Shimomura \cite{ShimoWS,ShimoMS}, and references therein. The first work that deals with scattering for the Cauchy problem of the Zakharov system (or any other Wave-\S systems in $3$ dimensions) is by Guo and Nakanishi \cite{GN}, where they considered small radial solutions in the energy space. 
In \cite{HPS}, the second author, Hani and Shatah proved pointwise decay and scattering for sufficiently smooth and localized solutions of \eqref{Z}.
Global dynamics below the ground state, under radial symmetry, have been analyzed in \cite{GNW2}.
The results in \cite{GN} and \cite{HPS} were then strengthened in \cite{GLNW}, where the authors 
used a generalized Strichartz estimate to obtain scattering for data in the energy space with additional angular regularity.

\subsubsection{Parameter range; The Yukawa System} The restriction $0<\gamma \leq 1$ in Theorem \ref{maintheo} is due to our methods,
but it is conceivable that similar techniques can apply for some $\gamma \le 0$\footnote{One 
important issue that arises in the case $\g < 0$ is the presence of a ``singularity'' in the nonlinear term of the second equation
of \eqref{ZSg0}, which is apparent when \eqref{ZSg0} is written as the first order system \eqref{Z_0} below.
While some of the arguments we present are still valid in this case, some others break down and would require substantial modifications
to yield the same final result.}.
We note that in the case $\gamma=-1$ (the Yukawa Wave-\S system), 
the expectation is to have modified scattering, 
that is, the behavior of the solutions for large $t$ does not coincide with that of linear solutions.
This was proven to be the case for the final value problem in \cite{GV-WS1,GV-WS2,GV-WS3}, \cite{ShimoWS} and \cite{GV-WSrev}.
Because of this expectation, we decide here to pursue a proof of global existence and scattering for \eqref{ZSg0}
based on the use of weighted spaces, rather than the approach of \cite{GN} based on radial Strichartz estimates\footnote{We 
believe it is likely that the techniques used in \cite{GN} apply to \eqref{ZSg0} in the parameter range 
that we consider here, to obtain scattering for small data in the energy space under the assumption of radial symmetry.}.
Indeed, our approach allows us to extract more precise information about the pointwise time decay of solutions,
which is an essential part of the analysis when dealing with nonlinear asymptotic behavior\footnote{For
some examples of modified scattering results in weighted spaces, 
see the papers \cite{HN,HNKdV,HNBO,KP,FNLS} which deal with semilinear equations,
\cite{Tsutsumi,ShimoMS} for results on the final value problem for field equations with long-range potentials, e.g. Maxwell-Schr\"odinger, 
and \cite{2dWW} for a recent example involving a quasilinear system, the water waves equations.}.
In the future we hope to refine our techniques, and possibly combine them with some of the recent advances in the area, see \cite{IP,IP2} for example,
to further push the admissible range of $\g$ towards $-1$.

\subsubsection{Techniques}
We briefly discuss the technical features of our argument. For a more detailed discussion, see Section \ref{secstrategy}. 
The strategy follows the general scheme of much recent work on small global solutions of dispersive systems, see for example \cite{GMS1,GMS2,GNT1,IP},
and \cite{HPS} which is more closely related to the problem we are considering. 
The vector fields method of Klainerman \cite{K0} cannot be applied 
to deal with \eqref{ZSg0}, because of the lack of space-time transformations leaving the combined Schr\"odinger and wave equations invariant. 
On the other hand, we use the observation, which appeared in \cite{HPS}, that $\Delta$ in the right hand side of second equation of \eqref{Z},
plays the role of a Klainerman null form \cite{K1}, allowing us to integrate by parts to gain decay\footnote{The
integration by parts argument in Fourier space is related to the space-time resonance method \cite{GMS1,GMS2},
and was used to deal with wave equations satisfying a nonresonance condition, akin to Klainerman's null condiditon \cite{K1}, in \cite{PS}.};
see Section \ref{secpre} and the identity \eqref{symg0}.
We show how this type of argument can be still used for \eqref{ZSg0}, where $\g>0$ can be arbitrarily small, 
if one combines it with a careful exploitation of the improved low frequency behavior of solutions of the linear wave equation.
An important role is also played by the use of the pseudo-scaling identity \eqref{dxiphi00} 
which allows to integrate by parts and estimate weighted norms of the Schr\"{o}dinger component.
Another key point is the spatial profile decomposition, also used in \cite{HPS}, to obtain decay for the wave component.

\section{Preliminary setup}\label{secpre}
Writing $w_{\pm}=\Lambda^{-1} (i\partial_t\pm \Lambda) n $, the system \eqref{Z} becomes
\begin{equation}
\label{Z_0}
\tag{Z$_0$}
\left\{
\begin{array}{l}
i\partial_t u +\Delta u = \frac{1}{2}(w_+ u - w_- u)
\\
\\
i\partial_t w_\pm \mp \Lambda w_\pm = \Lambda^\g {|u|}^2  \, .
\end{array}
\right.
\end{equation}
Let  $f = e^{-it\Delta} u$ and $g_\pm=e^{\pm it\Lambda} w_\pm$ denote the profiles, 
and let $\hat f = \FF f$ and $\hat g =\FF g$ denote  their Fourier transforms.
Duhamel's formula in Fourier space then reads
\begin{subequations}
\label{eq:profile}
\begin{align}
\label{inteqf}
&\hat f(\xi,t) = \hat f(\xi,0) + \sum_{\pm} \mp i\int_0^t \int_{\R^3} e^{is \phi_{\pm}(\xi,\eta)}\hat f(\xi-\eta,s) \hat g_\pm(\eta,s) 
  \mathrm{d}\eta \mathrm{d}s
\\
\label{inteqg}
&\hat g_\pm(\xi,t)  = \hat g_\pm (\xi,0)  - i\int_0^t \int_{\R^3} {|\xi|}^\g e^{is \psi_{\pm}(\xi,\eta)}
			\hat f(\xi-\eta,s) \overline{\hat{f}} (\eta,s) \mathrm{d}\eta \mathrm{d}s \, ,
\end{align}
\end{subequations}
where the phases are
 \begin{subequations}
\begin{align}
\begin{split}
\label{phi}
& \phi_\pm(\xi, \eta)  = {|\xi|}^2 - {|\xi-\eta|}^2 \pm |\eta| = 2\xi \cdot \eta -|\eta|^2 \pm |\eta| 
\end{split}
\\[.5 em]
\label{psi}
& \psi_{\pm}(\xi,\eta)  = \mp |\xi| - {|\xi-\eta|}^2 + {|\eta|}^2 = \mp |\xi| - {|\xi|}^2 + 2 \xi \cdot \eta \, .
\end{align}
\end{subequations}
We define the functions $F_{\pm}$ and $G_{\pm}$ by,
\begin{align}
\label{F}
F_\pm (\xi,t) & \overset{def}{=}  \FF^{-1} \int_0^t \int_{\R^3} e^{is \phi_{\pm}(\xi,\eta)}\hat f(\xi-\eta,s) \hat g_\pm(\eta,s) 
  \mathrm{d}\eta \mathrm{d}s \, ,
\\
\label{G}
G_\pm (\xi,t) & \overset{def}{=} \FF^{-1} \int_0^t \int_{\R^3} {|\xi|}^\g e^{is \psi_{\pm}(\xi,\eta)}\hat f(\xi-\eta,s) \overline{\hat{f}} (\eta,s) 
\mathrm{d}\eta \mathrm{d}s \, .
\end{align}



\subsection{Norms and a priori bounds}\label{secnorms}
We denote by $\dot{B}^s_{p, q}$ the Besov space defined by the norm 
\begin{equation}\label{Besov}
\|u\|_{\dot{B}^s_{p,q}}:=\left\| 2^{sk}\|P_k u\|_{L_x^p(\R^3)}\right\|_{l_k^q(\Z)}
\end{equation}
where $P_k$ denotes the Littlewood-Paley projection onto frequencies $|\xi|\sim 2^k$.

Given $\g >0$, we choose $\delta,\alpha>0$ and $N\gg1$ such that\footnote{Note 
that the first inequality in \eqref{eqn: alpha gamma} places a greater restriction on the size of $\alpha$ precisely when $0<\gamma<1/3$,
whereas for $\gamma\geq1/3$ we can choose $\alpha>0$ arbitrarily close to $1/6$.}
\begin{align}
\nonumber
& 5N^{-1} \leq \d \quad , \quad  \d \ll 1 \, ,
\\
\label{eqn: alpha gamma}
& 3\delta < \alpha \leq \gamma/2 - 10\delta \quad , \quad \alpha \leq 1/6 - 10\delta \, .
\end{align}
The proof of Theorem \ref{maintheo} follows a bootstrap argument in the Banach space $X$ defined by the norm\footnote{Local 
existence in time of solutions belonging to weighted Sobolev spaces can be established by standard techniques.}
\footnote{Without loss of generality we can restrict our attention to times $t \geq 1$.}:
\begin{align}
\label{norm}
\begin{split}
{\|(u,w_\pm)\|}_{X}  \overset{def}{=}  \sup_t   \left( t^{-\d}  {\|  f(t)  \|}_{H^{N+1}} \right. & +  t^{-\d}{\| xf (t) \|}_{L^2} 
  +  t^{-1+2\a+\d}  {\| {|x|}^2 f (t) \|}_{L^2}
\\
& + \, {\|  g_\pm(t) \|}_{H^{N}} 
+ t \left. {\|  e^{\mp it\Lambda} g_\pm (t) \|}_{\dot{B}^0_{\infty,1} }   \right) \, .
\end{split}
\end{align}
We choose $\e_1 = \e_0^{2/3}$ and assume a priori bounds on the quantities appearing in the $\|\cdot\|_{X}$ norm:
\begin{equation}
\label{apriorif}
 \quad  {\| f(t) \|}_{H^{N+1}}  \leq \e_1 t^\d  \, , 
\quad
{\| x f(t) \|}_{L^2} \leq \e_1 t^\d \, , 
\quad{\| {| x |}^2 f(t) \|}_{L^2} \leq \e_1 t^{1 - 2\a - \d}  \, , 
\end{equation}
and
\begin{align}
\label{aprioriG1}
& {\|  g_\pm(t) \|}_{H^N}  \leq \e_1 \, ,
\quad
{\| e^{\mp it\Lambda} g_\pm(t) \|}_{ \dot{B}^0_{\infty,1} } \leq \e_1 t^{-1}  \, .
\end{align}
As an intermediate step, we include the additional a priori bounds for $G_{\pm}$:
\begin{align}
\label{aprioriG2}
& {\| e^{\mp it\Lambda} x \Lambda G_\pm(t) \|}_{ L^{4/(1+\g)} } \leq \e_1 t^{- 1/4 + 3\g/4 -2\a -3\d}  \, , 
\\
\label{aprioriG3}
& {\| e^{\mp it\Lambda} \Lambda^{-1} G_\pm(t) \|}_{ L^3 } \leq \e_1 t^{-2\a - 3\d} \, ,
\quad
{\| \Lambda^{1/2} x G_\pm(t) \|}_{L^2} \leq \e_1 \, . 
\end{align} 

\begin{remark}
In contrast to \cite{HPS}, we do not place an a priori bound on $x^2G_{\pm}(t)$. Instead, we make greater use of the linear dispersive estimates for the wave group, and place an a priori bound on $e^{\mp it\Lambda} x \Lambda G_{\pm}(t)$ and $e^{\mp it\Lambda} \Lambda^{-1} G_{\pm}(t)$ in suitable $L^p$ spaces. This greatly simplifies many of the estimates in \cite{HPS}, while yielding the same exact conclusions for $\g \geq 1/3$.
\end{remark}

To obtain our result we will then show
\begin{align*}
 {\|(F_\pm, G_\pm)\|}_{X} \lesssim \e_1^2 \, ,
\end{align*}
which, together with the initial assumptions \eqref{data}-\eqref{data2} (see also \eqref{decaywave0} below), will give
\begin{align*}
 {\|(u, w_\pm)\|}_{X} \lesssim \e_0 + {\|(F_\pm, G_\pm)\|}_{X} \lesssim \e_0 + \e_1^2 \lesssim \e_0 + \e_0^{4/3} \, ,
\end{align*}
and guarantee a global-in-time solution belonging to $X$, provided $\e_0$ is chosen small enough.


\begin{remark}[Linear dispersive estimates]
From the linear estimates for the Schr\"odinger group
\begin{align}\label{disp}
& {\|e^{it\Delta} f\|}_{L^6} \lesssim  \frac{1}{t} {\| x f \|}_{L^2}
\quad , \quad 
  {\| e^{it\Delta} f \|}_{L^\infty} \lesssim  \frac{1}{t^{\frac 32}} {\| x f \|}^{\frac 12}_{L^2} {\| x^2 f \|}^{\frac 12}_{L^2} \, ,
\end{align}
we deduce  that the $X$ norm bounds
\begin{align}
\label{SL^6}
{\| e^{it\Delta} f \|}_{L^6}  & \lesssim \frac{1}{ { t }^{1-\d}} {\| u \|}_X \, ,
\\
\label{decayu}
{\| e^{it\Delta} f \|}_{L^\infty} & \lesssim \frac{1}{ { t }^{1+\a}} {\| u \|}_X  \, .
\end{align}
Moreover, by the linear dispersive estimate for the wave equation
\begin{equation}
\label{linearwave0}
{\| e^{it\Lambda} h \|}_{\dot{B}^0_{p,r}} \lesssim \frac{1}{t^{1-2/p}} {\| h \|}_{ {\dot{B}^{2(1-2/p)}_{{p^\prime},r}} } \quad , \quad p \geq 2 \, ,
\end{equation}
(cf. for example \cite{SS}), and the fact that $g_\pm(0) = \Lambda^{-1} i n_1 \pm n_0$, 
we see that \eqref{data2} implies
\begin{equation}
\label{decaywave0}
{\| e^{\mp i t\Lambda} g_{\pm}(0) \|}_{ \dot{B}^0_{\infty,1} } \lesssim \frac{\e_0}{t} \, .
\end{equation}
Finally, we note that \eqref{linearwave0} with $r = 2$, and embeddings between Besov and Sobolev spaces, 
gives
\begin{equation}
\label{linearwave}
{\| e^{it\Lambda} h \|}_{L^p} 
    \lesssim \frac{1}{ t^{1-2/p} } {\left\|  \Lambda^{2(1-2/p)} h \right\|}_{ L^{p^\prime} } \, .
\end{equation}
\end{remark}

\subsection{Strategy of the proof}\label{secstrategy}
From the definition of the $X$-norm, we see that in order to close our argument we need to obtain estimates on
high Sobolev norms of $(u,w_\pm)$ and weighted norms of $f$, and pointwise bounds for $w_\pm$.
Bounds on high Sobolev norms follow via standard energy estimates. 
To show decay for $w_\pm$ we use the weighted $L^2$ bounds on $f_\pm$.
To eventually bound weighted norms of $f_\pm$ we use the intermediate estimates \eqref{aprioriG2}-\eqref{aprioriG3} on $G_\pm$.
A key aspect is that the system \eqref{eq:profile} has null resonances, which we can use in combination with the space time resonance method 
to obtain weighted and pointwise bounds.

\subsubsection{\it Estimates for $G_\pm$}
We notice that the phase $\psi_{\pm}(\xi,\eta)$ satisfies,
\begin{equation}
\label{symg0}
|\xi|=\frac{1}{2}\frac{\xi}{|\xi|} \cdot \nabla_\eta \psi_\pm(\xi,\eta) \, .
\end{equation}
This means that the factor of $|\xi|^{\gamma}$ in \eqref{G} gives the equation a resonant structure,
although we see that this becomes weaker as $\gamma$ decreases towards $0$. 
We can thus use \eqref{symg0} to integrate by parts in $\eta$ and gain decay in $s$. 
In particular, using this together with Sobolev embedding and the a priori bounds on $f$, 
we can obtain a uniform in time estimate for $x\Lambda^{1/2}G_{\pm}$ in $L^2$. 
Carefully exploiting the linear dispersive estimates for wave group gives the estimates for
$e^{\mp it\Lambda}x\Lambda G_{\pm}$ and $e^{\mp it\Lambda}\Lambda^{-1} G_{\pm}$. 
These estimates are presented in Section \ref{secG}.

\subsubsection{$L^\infty$ bounds}
These are obtained similarly to \cite{HPS}.
The pointwise decay of $e^{it\Delta} f$, see \eqref{decayu}, is a direct consequence of the weighted estimates for $xf$ and $|x|^2f$.
To obtain a $t^{-1}$ pointwise decay for $e^{\mp it\Lambda} G_\pm$,
we use the improved small frequency behavior of solutions of the linear wave equation, see \eqref{linearwave0}, the identity \eqref{symg0}, and 
a similar argument to \cite{HPS}. Some of the details for the estimate of $e^{\mp it\Lambda}G_\pm$ are presented in Section \ref{secdecayw}.

\subsubsection{Weighted $L^2$ estimates for $F_\pm$}  
To obtain the desired weighted estimates for $F_{\pm}$, we use the (pseudo-scaling) identity
\begin{equation}
\label{dxiphi00}
\nabla_\xi \phi_\pm = - 2\eta = - 2 \frac{\eta}{|\eta|} \left( \frac{\eta}{|\eta|}\cdot \nabla_\eta \phi_\pm \right) 
  - 2\frac{\phi_\pm}{|\eta|}\frac{\eta}{|\eta|} \, .
\end{equation}
More specifically, calculating $xF_{\pm}$ in Fourier space involves applying a derivative in $\xi$ to $\hat{F}_{\pm}$.
When this derivative is applied to the factor of $e^{is\phi_{\pm}}(\xi,\eta)$ we obtain a factor of $\nabla_\xi \phi_\pm (\xi,\eta)$.
We can then use \eqref{dxiphi00} to express $\nabla_\xi \phi_\pm (\xi,\eta)$ in terms of $\nabla_\eta \phi_\pm(\xi,\eta)$ and $\phi_{\pm}(\xi,\eta)$.
This allows us to integrate by parts one time in both space and time, and the estimate for $xF_{\pm}$ in $L^2$
then follows from the a priori estimates for $f$ and $G_{\pm}$.

The equality in \eqref{dxiphi00} is also used in the estimate of $x^2F_{\pm}$. 
Applying two derivatives in $\xi$ to $\hat{F}_{\pm}(\xi,s)$, we find that one term contains a factor of $(\nabla_\xi \phi_\pm(\xi,\eta))^2$.
If, as in \cite{HPS}, we use \eqref{dxiphi00} to integrate by parts twice, we end up with a term containing a factor of $x^2G_{\pm}$.
However, we no longer have an a priori estimate on $x^2G_{\pm}$, so we do not proceed in this way. Instead, for this term in $x^2F_{\pm}$,
we only integrate by parts in $\eta$ once, 
and make use of the $L^p$ estimates \eqref{aprioriG2}-\eqref{aprioriG3} 
on $e^{\mp it\Lambda}x\Lambda G_{\pm}(t)$ and $e^{\mp it\Lambda}\Lambda^{-1}G_{\pm}(t)$.
We carry out these estimates in Section 4.

\subsection{Energy Estimates and high frequency cutoff}\label{secenergy}
We have the following:
\begin{proposition}
\label{proenergy}
Let $F_\pm$ and $G_\pm$ be given by \eqref{F} and \eqref{G} respectively. Then, for ${\| (u, w_\pm) \|}_X \leq \e_1$, we have
\begin{equation*}
 {\| G_\pm (t) \|}_{H^{N}}  + t^{-\d} {\| F_\pm (t) \|}_{H^{N+1}} \lesssim \e_1^2 \, .
\end{equation*}
\end{proposition}
We do not provide details on how to obtain the above bounds, since they are fairly easy to show and can be proved as in \cite{HPS}. 

We can use the a priori bounds on high Sobolev norms to reduce all of our estimates 
to frequencies smaller than $s^{\delta_N}$, where $\delta_N \ll 1$ is chosen small depending on $N$.
To see this, let us assume in what follows that at least one of the frequencies $\eta$ or $\xi-\eta$ in the expressions for $F_\pm$ and $G_\pm$,
see \eqref{F} and \eqref{G}, has size larger than $s^{2/(N-2)}$.

Observe that for all $k \geq 0$ one has ${\| P_{\geq k} v (s) \|}_{L^2} \lesssim 2^{-k l } {\| v(s) \|}_{H^l}$,
and therefore, for frequencies $2^k \gtrsim s^{ 2/(N-2) } \gtrsim 1$, we have from the apriori assumptions \eqref{apriorif}-\eqref{aprioriG1},
\begin{align}
\label{P_k}
\begin{split}
{\| P_{\geq k} u (s) \|}_{H^3} & \lesssim  2^{-k (N-2) } {\| u(s) \|}_{H^{N+1}} \lesssim \e_1 s^{-2 + \d}
\\
{\| P_{\geq k} w_\pm (s) \|}_{H^2} & \lesssim  2^{-k (N-2) } {\| w_\pm (s) \|}_{H^{N}} \lesssim \e_1 s^{-2} \, .
\end{split}
\end{align}
These estimates are already sufficient to bound all norms that do not involve weights.

To establish bounds on $L^p$ norms involving weights, we are going to apply $\nabla_\xi$ to the bilinear terms $\widehat{F}_\pm$ and $\widehat{G}_\pm$.
We start by noticing that the action of weights on Littlewood-Paley projections (on high-frequencies) is harmless,
and only gives terms that are easier to treat than any other term that we will have to deal with.
We now briefly discuss how to estimate all the other contributions that result from applying derivatives to $\widehat{F}_\pm$ and $\widehat{G}_\pm$ 
in the case when $\max\{|\eta|,|\xi-\eta|\} \gtrsim s^{2/(N-2)}$.

\subsubsection*{High frequency contributions in $F_\pm$}
Looking at \eqref{F}, we see that applying $\nabla_\xi$ twice to $\widehat{F}_\pm$ gives three types of contributions.
The first are those where $\nabla_\xi^2$ hits the phase $e^{is\phi_\pm}$: 
these terms will contain powers of $s$ but will not involve weights on the inputs $f$ and $g$.
Since at least one frequency has size larger than $s^{2/(N-2)}$, then all these terms can be estimated directly using \eqref{P_k}.
The second type of terms that arise are those with $\nabla_\xi^2$ hitting the input $\widehat{f}(\xi-\eta)$.
If this happens, the same estimates that we are going to perform  below in section \ref{secF} will work, regardless of the size of frequencies.
The third type of contribution is where one $\nabla_\xi$ hits the phase $e^{is\phi_\pm}$, and the other hits the input $\widehat{f}(\xi-\eta)$.
In this case, if $\eta$ is the largest frequency, the term can be estimated directly using the second bound in \eqref{P_k}.
If instead $\xi-\eta$ is the largest frequency, one can write $\nabla_\xi \widehat{f}(\xi-\eta) = \nabla_\eta \widehat{f}(\xi-\eta)$
and integrate by parts in $\eta$. 
This will generate one term with losses in powers of $s$ similar to the first type of term discussed above, 
plus a term where the inputs are $\widehat{f}(\xi-\eta)$ and $\nabla_\eta \widehat{g}_\pm (\eta)$. 
Since we are assuming that $|\xi-\eta| \gtrsim s^{ 2/(N-2)}$, we can use Sobolev's embedding, the second apriori assumption in \eqref{aprioriG3},
and the first inequality in \eqref{P_k}, to estimate directly this term 
and obtain the desired bound without resorting to further manipulations.

\subsubsection*{High frequency contributions in $G_\pm$}
We now briefly describe how to obtain the bound \eqref{aprioriG2} and the second bound in \eqref{aprioriG3}
for $G_\pm$, in the case of high frequencies.
Since the inputs in \eqref{G} are symmetric, up to complex conjugation (which leaves our norms invariant),
we can assume that $|\eta| \gtrsim |\xi-\eta|$ and $|\eta| \gtrsim s^{ 2/(N-2)}$.
When applying derivatives to $\widehat{G}_\pm$ we then obtain two types of contributions, similar to the ones discussed in the previous paragraph.
The first contribution is the one where $\nabla_\xi$ hits the phase $e^{is\psi_\pm}$. 
This will cause a loss of a power of $s$ which can be overcome directly using the decay given by \eqref{P_k}.
The second type of term will contain $\nabla_\xi \widehat{f}(\xi-\eta)$ as an input.
Since we have already reduced ourselves to the case when $\eta$ is the largest frequency, we can again use \eqref{P_k},
and the apriori bound \eqref{apriorif} on $xf$, to get the desired estimates in a straightforward fashion.

The above discussion shows that in estimating weighted norms of the bilinear terms $F_\pm$ and $G_\pm$ in \eqref{F} and \eqref{G},
we can always reduce our analysis to frequencies $|\xi-\eta| , |\eta| \lesssim s^{2/(N-2)}$,
for otherwise all the desired bounds can be shown to hold true without too much effort. 
We therefore agree on the following:

\begin{convention}\label{convention1}
In the rest of the paper, we assume that all frequencies, $\xi-\eta$ and $\eta$, appearing in the 
estimates of the bilinear terms \eqref{F} and \eqref{G}, have size bounded above by $s^{\delta_N}$, where  $\delta_N := \frac{2}{N-2}$
and the integer $N \gg 1$ is determined in the course of our proof by several upperbounds on $\d_N$.
In particular, expressions such as $|\xi|$ or $\nabla_\xi \psi_\pm (\xi,\eta)$ will be constantly replaced by a factor of $s^{\delta_N}$.
\end{convention}

We will also adopt the following additional notational convention:
\begin{convention}
To make notations lighter, we will often drop the $\pm$ indices, and omit the dependence on the time $t$.
Moreover, in the estimates of the bilinear terms $F$ and $G$ in \eqref{F} and \eqref{G},
we will often only consider the contribution of the integrals from $1$ to $\infty$.
All of the contributions coming from integrating between $0$ and $1$ are bounded in a straightforward fashion
by Sobolev's embedding, and our control of high Sobolev norms of the solution $(u,w)$.
\end{convention}

\section{Estimates for $G$}\label{secG}
We recall that we have the following a priori assumptions on $f=e^{it\Delta}u$:
\begin{equation} \label{eqn: f}
 {\| xf \|}_{L^2}\leq \e_1 t^\delta \quad , \quad {\|x^2f\|}_{L^2} \leq \e_1 t^{1-2\alpha-\delta} 
  \quad , \quad {\|f\|}_{H^N} t^{\delta} \leq \e_1 \, .
\end{equation}
As a consequence, the following dispersive bounds for $u$ hold:
\begin{equation} \label{eqn: u}
{\|u \|}_{L^{\infty}} \lesssim \e_1 t^{-1-\alpha} \quad, \quad {\|u\|}_{L^6} \lesssim \e_1 t^{-1+\delta} \, . 
\end{equation}

\subsection{Weighted estimates for $G$}\label{secGL2}
In this section we are going to prove the following:
\begin{proposition}
Let $G$ be the bilinear term defined in \eqref{G}:
\begin{equation}
 \label{eqn: G}
G = \FF^{-1} \int_1^t \int_{\R^3}|\xi|^{\gamma} e^{is \psi(\xi,\eta)}\hat f(\xi-\eta,s) \overline{\hat{f}} (\eta,s) \,\mathrm{d}\eta \mathrm{d}s \, .
\end{equation}
Under the a priori assumptions \eqref{eqn: f} and \eqref{eqn: u}, we have
\begin{align*}
& {\| \Lambda^{1/2}x G \|}_{L^2} \lesssim \e_1^2 \, ,
\\
& {\| e^{it\Lambda}x\Lambda G \|}_{L^{4/(1+\gamma)}} \lesssim t^{-1/4 + 3\gamma/4 - 2\alpha - 3\delta} \e_1^2 \, ,
\\
& {\| e^{it\Lambda}\Lambda^{-1} G \|}_{L^3} \lesssim t^{- 2\alpha - 3\delta} \e_1^2 \, .
\end{align*}
\end{proposition}

The proof of the above proposition is split into Lemma \ref{prop: G1}, \ref{prop: G2} and \ref{prop: G3} below.
We recall the following assumptions on the relative sizes of $\gamma$ and $\alpha$ from \eqref{eqn: alpha gamma},
 \begin{equation} \label{eqn: alpha}
  3\delta < \alpha \leq \gamma/2 - 10\delta \quad , \quad \alpha \leq 1/6 - 10\delta.
 \end{equation}

\begin{lemma}
\label{prop: G1}
Let $G$ be the bilinear term defined in \eqref{eqn: G} and let $\alpha$ satisfy \eqref{eqn: alpha}. Then, 
\begin{equation*}
  {\| \Lambda^{1/2}x G \|}_{L^2} \lesssim \e_1^2 \, .
\end{equation*}
\end{lemma}
To prove this, it is crucial to notice that low frequencies play the role of a special null resonant structure in the nonlinear term $G$,
see \eqref{symg0},
\begin{equation}
\label{symg}
|\xi|=\frac{1}{2}\frac{\xi}{|\xi|} \cdot \nabla_\eta \psi \, .
\end{equation}

\begin{proof}[Proof of Lemma \ref{prop: G1}]
Applying $|\xi|^{1/2}\nabla_\xi$ to $\hat{G}$ gives the terms:
\begin{align}
\label{xi^1/2dxig1}
& \int_1^t \int_{\R^3} e^{is \psi (\xi,\eta)} |\xi|^{1/2+\gamma} \nabla_\xi \hat{f}(\xi-\eta,s) \overline{\hat{f}} (\eta,s) \mathrm{d}\eta \mathrm{d}s
\\
\label{xi^1/2dxig2}
& \int_1^t \int_{\R^3} s \nabla_\xi \psi  e^{i s \psi (\xi,\eta)} |\xi|^{1/2+\gamma} \hat{f}(\xi-\eta,s) \overline{\hat{f}} (\eta,s) \mathrm{d}\eta \mathrm{d}s \, ,
\end{align}
plus an easier term when $\nabla_\xi$ hits the symbol $|\xi|^{\gamma}$.
\eqref{xi^1/2dxig1} is easily estimated by H\"{o}lder's inequality together with the estimates on $f$ and $u$ in \eqref{eqn: f} and \eqref{eqn: u},
and using the first condition in \eqref{eqn: alpha}:
\begin{align*}
{\| \eqref{xi^1/2dxig1} \|}_{L^2} & \lesssim \int_1^t s^{2\delta_N} {\|  x f \|}_{L^2} {\|  e^{is\Delta} f  \|}_{L^\infty} \, \mathrm{d}s
\lesssim \int_1^t s^{2\delta_N} s^\d \frac{1}{s^{1+\a}} \, ds \lesssim 1 \, .
\end{align*}

Using the identity \eqref{symg}
and integrating by parts in $\eta$, \eqref{xi^1/2dxig2} gives terms of the form
\begin{align*}
 \int_1^t \int_{\R^3} e^{i s\psi (\xi,\eta)} m_1(\xi,\eta)|\xi|^{-1/2+\gamma} \nabla_\eta \hat{f}(\xi-\eta,s) \overline{\hat{f}} (\eta,s) \mathrm{d}\eta ds,
\end{align*}
together with symmetric or easier terms. Here $m_1(\xi,\eta)$ is a symbol satisfying homogeneous bounds of order 1 for large frequencies, and is otherwise harmless. By Plancherel, the $L^2$-norm of this term is bounded by
\[ \int_1^t s^{\delta_N} \|\Lambda^{-1/2+\gamma} e^{is\Delta}(xf)e^{is\Delta}f\|_{L^2} \, \mathrm{d}s. \]
For $1/2\leq\gamma<1$, we can estimate this by,
\begin{align*}
\int_1^t s^{3\delta_N/2} \|e^{is\Delta}(xf)e^{is\Delta}f\|_{L^2} \, ds \lesssim  \int_1^t s^{3\delta_N/2} \|xf\|_{L^2} \|e^{is\Delta}f\|_{L^\infty} \, \mathrm{d}s \\
 \lesssim \int_1^t  s^{3\delta_N/2+\delta}s^{-1-\alpha} \, \mathrm{d}s \lesssim 1,
\end{align*}
since $\alpha>3\delta$.
For $0\leq \gamma<1/2$, applying Sobolev embedding, we can estimate this by
\begin{align*}
\int_1^t s^{\delta_N} \| e^{is\Delta}(xf)e^{is\Delta}f\|_{L^{3/(2-\gamma)}} \, \mathrm{d}s \lesssim \int_1^t s^{\delta_N} \|xf\|_{L^2}\|e^{is\Delta}f\|_{L^{6/(1-2\gamma)}} \, \mathrm{d}s\\
 \lesssim \int_1^t s^{\delta_N} s^\delta \|e^{is\Delta}f\|_{L^{6/(1-2\gamma)}} \, \mathrm{d}s.
 \end{align*}
Interpolating between the $L^6$ and $L^\infty$ estimates on $u = e^{it\Delta}f$ from \eqref{eqn: u}, 
and choosing $\delta>0$ sufficiently small, we can ensure that this integral has an $O(1)$ bound.
\end{proof}

\begin{lemma}
\label{prop: G2}
Let $G$ be the bilinear term defined in \eqref{eqn: G} and let $\alpha$ satisfy \eqref{eqn: alpha}. Then, 
\begin{equation*}
  {\| e^{it\Lambda}x\Lambda G \|}_{L^{4/(1+\gamma)}} \lesssim t^{-1/4 + 3\gamma/4 - 2\alpha - 3\delta} \e_1^2 \, .
\end{equation*}
\end{lemma}
\begin{proof}

Applying $\nabla_\xi|\xi|$ to $\hat{G}$ gives the terms:
\begin{align}
\label{dxixig1}
& \int_1^t \int_{\R^3} e^{is \psi (\xi,\eta)} |\xi|^{1+\gamma} \nabla_\xi \hat{f}(\xi-\eta,s) \overline{\hat{f}} (\eta,s) \mathrm{d}\eta \mathrm{d}s
\\
\label{dxixig2}
& \int_1^t \int_{\R^3}s \nabla_\xi \psi  e^{is \psi (\xi,\eta)} |\xi|^{1+\gamma} \hat{f}(\xi-\eta,s) \overline{\hat{f}} (\eta,s) \mathrm{d}\eta \mathrm{d}s \, ,
\end{align}
plus an easier term when $\nabla_\xi$ hits the symbol $|\xi|^{1+\gamma}$.

We now use \eqref{symg} and integrate by parts in $\eta$ to write \eqref{dxixig2} as terms of the form
\begin{equation} \label{dxixig3}
 \int_1^t \int_{\R^3} e^{i s\psi (\xi,\eta)} m_1(\xi,\eta) |\xi|^\gamma\nabla_\eta \hat{f}(\xi-\eta,s) \overline{\hat{f}} (\eta,s) \, \mathrm{d}\eta \mathrm{d}s,
\end{equation}
together with symmetric or easier terms. 
Here, as before, $m_1(\xi,\eta)$ is a symbol satisfying homogeneous bounds of order 1 for large frequencies, and is otherwise harmless. 

Using the linear dispersive estimate, the contribution from \eqref{dxixig1}-\eqref{dxixig2} 
can thus be bounded by
\begin{align*} 
& \int_0^t\frac{1}{(t-s)^{1/2-\gamma/2}} s^{\delta} \| \Lambda^{-\gamma}\Lambda^{\gamma}e^{is\Delta}(xf)(e^{is\Delta}f)\|_{L^{4/(3-\gamma)}} \, \mathrm{d}s 
\\
& \lesssim  \int_0^t\frac{1}{(t-s)^{1/2-\gamma/2}} s^{\delta} \|xf\|_{L^2} \|e^{is\Delta}f\|_{L^{4/(1-\gamma)}} \, \mathrm{d}s 
\\
& \lesssim \int_0^t\frac{1}{(t-s)^{1/2-\gamma/2}} s^{2\delta}s^{-(3/4+3\gamma/4)(1-\delta)} \, \mathrm{d}s.
 \end{align*}
 If $0<\gamma<1/3$, then we can bound this integral by $t^{-1/4 - \gamma/4 +4\delta}$. 
If $1/3\leq \gamma <1$, then instead, we obtain a bound of $t^{-1/2+\gamma/2+ 4\delta}$. 
By the assumptions on $\alpha$ from \eqref{eqn: alpha}, these estimates are sufficient.
\end{proof}

\begin{lemma}
\label{prop: G3}
Let $G$ be the bilinear term defined in \eqref{eqn: G} and let $\alpha$ satisfy \eqref{eqn: alpha}. Then, 
\begin{equation*}
  {\| e^{it\Lambda}\Lambda^{-1} G \|}_{L^3} \lesssim \e_1^2 t^{- 2\alpha - 3\delta} \, .
\end{equation*}
\end{lemma}

\begin{proof}
Using \eqref{eqn: G}, we can write,
\[ e^{it\Lambda}\Lambda^{-1} G(t) = \int_1^t e^{i(t-s)\Lambda}\Lambda^{-1+\gamma}|u(s)|^2 \, \mathrm{d}s. \]
Thus, by the linear dispersive estimate,
\begin{equation} \label{e^itxi xi^-1g}
 {\| e^{it\Lambda}\Lambda^{-1} G \|}_{L^3} \lesssim \int_1^t \frac{1}{(t-s)^{1/3}} \|\Lambda^{-1/3+\gamma}G\|_{L^{3/2}} \, \mathrm{d}s.
\end{equation}
Suppose first that $\gamma\geq1/3$. Then, we can use the bounds from \eqref{eqn: f} to estimate \eqref{e^itxi xi^-1g} by
\[ \int_1^t \frac{1}{(t-s)^{1/3}}s^{\delta/N} \|u\|^2_{L^3} \, \mathrm{d}s \lesssim \int_1^t \frac{1}{(t-s)^{1/3}}s^{\delta/N}\frac{1}{s^{1-\delta}} \, \mathrm{d}s \lesssim \frac{1}{t^{1/3}}t^{2\delta}. \]
By the third assumption on $\alpha$ from \eqref{eqn: alpha}, this gives us the desired bound.
For $0<\gamma<1/3$, we first apply Sobolev embedding to estimate \eqref{e^itxi xi^-1g} by
\[  \int_1^t \frac{1}{(t-s)^{1/3}} \|u\|^2_{L^{18/(7-3\gamma)}} \, \mathrm{d}s. \]
Using the bounds from \eqref{eqn: u}, we thus obtain,
\[ \eqref{e^itxi xi^-1g} \lesssim \int_1^t \frac{1}{(t-s)^{1/3}}s^{-(2/3+\gamma)(1-\delta)} \, \mathrm{d}s \lesssim t^{-\gamma+2\delta}. \]
By the second assumption on $\alpha$ from \eqref{eqn: alpha}, this gives us the desired bound.
\end{proof}

\subsection{Decay estimate for $G$}\label{secGLinfty}
\label{secdecayw}
In this section we want to show the following:
\begin{proposition}
\label{prodecayw}
Let $G_\pm$ be the bilinear term defined in \eqref{G}.
Under the apriori assumptions \eqref{eqn: f} and \eqref{eqn: u} we have 
\begin{equation*}
{\| e^{it\Lambda} G_\pm \|}_{ \dot{B}^0_{\infty,1} } \lesssim \e_1^2 {(1+t)}^{-1} \, .
\end{equation*}
\end{proposition}
The proof of the above proposition is analogous to the one in section 6 of \cite{HPS}.
We provide some of the details below.

\nl
{\it Proof of Proposition \ref{prodecayw}.}
Let us split $G$ into two parts, depending on the localization of the inputs.
More precisely, we let
$G = G_1 + G_2$
where 
\begin{align*}
G_1 & :=   G(f_{\leq s^{1/8} } , \bar{f}) + G(f, \bar{f}_{\leq s^{1/8} })
\\
G_2 & :=   G(f_{\geq s^{1/8} } , \bar{f}_{\geq s^{1/8} } ) \, .
\end{align*}
The component $G_1$ can be shown to be bounded in a weighted Sobolev space stronger than $\dot{B}^{2}_{1,1}$;
this directly gives the desired bound on $e^{it\Lambda} G_1$.
The decay of $e^{it\Lambda} G_2$ will instead be proven using the null structure \eqref{symg} in conjuction with the
improved small frequency behavior of the dispersive estimate for linear wave equation.
We will crucially use the fact that the $L^2$ norm of $f_{\geq s^{1/8}}$ decays in $L^2$.
Since the two terms in the definition of $G_1$ are similar we can reduce to consider $G_1$ and $G_2$ given by
\begin{align}
\label{G_1}
\hat{G}_1 & =   \int_1^t \int_{\R^3} {|\xi|}^\g e^{is \psi (\xi,\eta)} \hat{f_{\leq s^{1/8} }} (\xi-\eta,s) \overline{ \hat{f} }(\eta,s) \, \mathrm{d}\eta \mathrm{d}s
\\
\label{G_2}
\hat{G}_2 & =   \int_1^t \int_{\R^3} {|\xi|}^\g e^{is \psi (\xi,\eta)} \hat{f_{\geq s^{1/8} }} (\xi-\eta,s) \overline{ \hat{f_{\geq s^{1/8} }} } (\eta,s)\, \mathrm{d}\eta \mathrm{d}s \, .
\end{align}

\nl
{\it Decay estimate for $e^{it\Lambda} G_1$}.
To show that $G_1$ is bounded in  $\dot{B}^{2}_{1,1}$ we will interpolate weighted $L^2$ norms inside the time integral.
One can then exploit the ``small'' support of $f_{\leq s^{1/8}}$ to get improvements on these weighted norms, 
and on the decay of $e^{is\Delta}f_{\leq s^{1/8}}$.
Recalling that we are only considering frequencies $k$ such that $2^k \leq s^{\d_N}$, we aim to prove
\begin{align*}
\int_1^t  \, \sum_{k = -\infty}^{ \log s^{\d_N} } 2^{2k} {\| P_k \Lambda^\g e^{-is\Lambda} 
				\left( e^{is\Delta} f_{\leq s^{1/8}} e^{-is\Delta} \bar{f} \right) \|}_{L^1} \, \mathrm{d}s  \lesssim 1 \, .
\end{align*}
Converting a factor of $2^{(2-\g)k}$ into derivatives $\Lambda^{2-\g}$, throwing away the projection $P_k$, and performing the sum,
we see that is suffices to show
\begin{align*}
\int_1^t  s^{\g \d_N} {\| \Lambda^2 e^{-is\Lambda} \left( e^{is\Delta} f_{\leq s^{1/8}} e^{-is\Delta} \bar{f} \right) \|}_{L^1} \, \mathrm{d}s  \lesssim 1 \, .
\end{align*}
Since ${\| \cdot \|}_{L^1} \lesssim {\| x \cdot \|}_{L^2}^{1/2} {\| x^2 \cdot \|}_{L^2}^{1/2}$, 
the above estimate will follow from the inequalities
\begin{subequations}
\begin{align}
\label{decayg_11}
& {\left\| |x| \Lambda^2 e^{-is\Lambda} \left( e^{is\Delta} f_{\leq s^{1/8}} e^{-is\Delta} \bar{f} \right) \right\|}_{L^2}  
  \lesssim  s^{-7/4} \, ,
\\
\label{decayg_12}
& {\left\| {|x|}^2 \Lambda^2 e^{-is\Lambda} \left( e^{is\Delta} f_{\leq s^{1/8}} e^{-is\Delta} \bar{f} \right) \right\|}_{L^2} 
  \lesssim s^{-1} \, .
\end{align}
\end{subequations}
These two estimates have been already proven in \cite{HPS} under the same apriori assumptions made in \eqref{eqn: u}.
Therefore, we omit them and refer the reader to section 6.1 of \cite{HPS} for a detailed proof.

\nl
{\it Decay estimate for $e^{it\Lambda} G_2$}.
We write 
\begin{align*}
e^{it\Lambda} G_2 (t,x) 
& = 
  \int_1^t e^{i(t-s)\Lambda} \Lambda^{\g-1} \FF^{-1}_\xi \left[ \int_{\R^3} |\xi| e^{is \tilde{\psi} (\xi,\eta)} 
  \hat{f_{\geq s^{1/8} }} (\xi-\eta,s) \overline{ \hat{f_{\geq s^{1/8} }} } (\eta,s) \mathrm{d}\eta  \right] \, \mathrm{d}s
\end{align*}
where $\tilde{\psi} (\xi,\eta) = {|\xi-\eta|}^2 - {|\eta|}^2 = {|\xi|}^2 - 2 \xi \cdot \eta$.
We now want to use \eqref{symg} to integrate by parts in $\eta$.
By symmetry we can reduce to consider the following term:
\begin{equation}
\label{g_2decay}
\int_1^t e^{i(t-s)\Lambda}  \frac{1}{s} \, \Lambda^{\g-1} \FF^{-1}_\xi \left[ \int_{\R^3} \frac{\xi}{|\xi|}
e^{is \tilde{\psi} (\xi,\eta)} \,  \nabla_\eta \hat{f_{\geq s^{1/8} }} (\xi-\eta,s) \overline{ \hat{f_{\geq s^{1/8} }} } (\eta,s) \mathrm{d}\eta \right] \mathrm{d}s \, .
\end{equation}
The contribution of the time integral between $t-1$ and $t$ can be easily estimated by Sobolev embedding. 
To estimate the contribution from $1$ to $t-1$, we use the linear dispersive estimate for the wave equation \eqref{linearwave0},
and our large frequency cutoff convention, to bound it by
\begin{align*}
&\int_1^{t-1} \frac{1}{t-s} \,  \frac{1}{s} \, 
  \sum_{k = - \infty}^{ \log s^{\delta_N} } 2^{(\g+1)k} {\left\|  
  P_k \left(  e^{is\Delta} x f_{\geq s^{1/8}}  \,  e^{is\Delta} f_{\geq s^{1/8}}  \right)  \right\|}_{L^1} \, ds
\\
 & \lesssim \int_1^{t-1} \frac{1}{t-s} \,  \frac{1}{s} \,  s^{2\delta_N}  
 	{\left\| e^{is\Delta} x f_{\geq s^{1/8}} \, e^{is\Delta} f_{\geq s^{1/8}} \right\|}_{L^1} \, ds
\\
 & \lesssim \int_1^{t-1} \frac{1}{t-s} \,  \frac{1}{s} \,  s^{2\delta_N} 
 	{\left\|  x f \right\|}_{L^2} {\| f_{\geq s^{1/8}}  \|}_{L^2} \, ds
  \lesssim \int_1^{t-1} \frac{1}{t-s} \,  \frac{1}{s} \,  s^{2\delta_N}  \, s^\d \frac{1}{s^{1/8}} s^\d  \, ds \lesssim \frac{1}{t} \, .
  \hskip 80pt \Box
\end{align*}

\newcommand{\ephi}{e^{is\phi (\xi,\eta)}}

\newcommand{\dds}{\,\mathrm{d}\eta\mathrm{d}s}

\newcommand{\fg}{\widehat{f}(\xi-\eta,s)\widehat{g} (\eta,s)}

\newcommand{\dirphi}{\frac{\eta}{|\eta|}\cdot \partial_{\eta} \phi}

\newcommand{\phiovereta}{\frac{\phi}{|\eta|}}

\section{Estimates for $F$}\label{secF}
Recall that we are making the following a priori assumptions on $g$ and $f$:
\begin{align}
\label{eqn: gL}
& {\| \Lambda^{1/2} x G\|}_{L^2} \leq \e_1,
\qquad {\|e^{it\Lambda} x\Lambda G\|}_{L^{4/(1+\gamma)}} \leq t^{-1/4+3\gamma/4-2\alpha-3\delta} \e_1,
\\
\label{eqn: Lminus1G3}
& {\|e^{it\Lambda} G\|}_{L^\infty}  \leq t^{-1} \e_1, 
  \qquad {\|e^{it\Lambda}\Lambda^{-1} G\|}_{L^3} \leq t^{- 2\alpha - 3\delta} \e_1,
\\
\label{apriorif10}
& {\| x f \|}_{L^\infty}  \leq t^{\d} \e_1,  \qquad {\| x^2 f \|}_{L^2} \leq t^{1-2\alpha-\delta} \e_1 .
\end{align}
In this section we want to establish estimates for $F$ defined as
\begin{equation}\label{eqn: Fdef}
\widehat{F}(\xi,t) = 
  \int_0^t\int_{\mathbb{R}^3} \ephi \fg \,\dds \, .
\end{equation}
Recall that 
\begin{equation}\label{eqn: null}
\partial_\xi \phi(\xi,\eta) = -2\eta = 
  -2\frac{\eta}{|\eta|}\left(\frac{\eta}{|\eta|}\cdot \partial_\eta \phi(\xi,\eta) \right)-2\frac{\phi}{|\eta|}\frac{\eta}{|\eta|} \, ,
\end{equation}
and note that
$\partial_{\xi_i}^2 \phi (\xi,\eta) = 0$,
which in particular leads to $\partial_{\xi^i}^2 \ephi = -4s^2\eta_i^2 \ephi$.


\subsection{Estimate for $x F$}
In this section we aim to prove the following lemma:
\begin{lemma}
Let $F$ be defined by \eqref{eqn: Fdef}. Under the apriori assumptions \eqref{eqn: gL}--\eqref{apriorif10}
we have
\begin{equation}
{\| x F \|}_{L^2} \lesssim  t^\delta \e_1^2 \, .
\end{equation}
\end{lemma}

\begin{proof}
We have that $\nabla_\xi \widehat{F}$ is given by a linear combination of terms of the form
\begin{align}
\label{eqn: xfsum1}
& \int_0^t\int_{\mathbb{R}^3}e^{is\phi(\xi,\eta)}\nabla_\xi\widehat{f}(\xi-\eta,s)\widehat{g}(\eta,s)\, \dds 
\\ 
\label{eqn: xfsum2}
& \int_0^t\int_{\mathbb{R}^3} s \nabla_\xi\phi \, e^{is\phi(\xi,\eta)}\widehat{f}(\xi-\eta,s) \widehat{g}(\eta,s)\, \dds.
\end{align}
Using \eqref{eqn: null} to integrate by parts in $\eta$ and $s$ in equation \eqref{eqn: xfsum2}, we have the following contributions:
\begin{align} 
\label{eqn: xf21}
& \int_0^t\int_{\mathbb{R}^3}e^{is\phi(\xi,\eta)}\nabla_\eta\widehat{f}(\xi-\eta,s)\widehat{g}(\eta,s)\, \dds
\\
\label{eqn: xf22}
& \int_0^t\int_{\mathbb{R}^3}e^{is\phi(\xi,\eta)}\widehat{f}(\xi-\eta,s)\nabla_\eta\widehat{g}(\eta,s)\, \dds
\\
\label{eqn: xf23}
& \int_{\mathbb{R}^3}te^{it\phi(\xi,\eta)}\widehat{f}(\xi-\eta,t)\frac{1}{|\eta|}\widehat{g}(\eta,t)\, \mathrm{d}\eta
\\
\label{eqn: xf24}
& \int_0^t\int_{\mathbb{R}^3}se^{is\phi(\xi,\eta)}\widehat{f}(\xi-\eta,s)\frac{1}{|\eta|}\partial_s\widehat{g}(\eta,s)\, \dds
\\
& \label{eqn: xf25}
\int_0^t\int_{\mathbb{R}^3}se^{is\phi(\xi,\eta)}\partial_s\widehat{f}(\xi-\eta,s)\frac{1}{|\eta|}\widehat{g}(\eta,s)\, \dds.
\end{align}
The terms \eqref{eqn: xfsum1} and \eqref{eqn: xf21} can be bounded as
\[ \|\eqref{eqn: xfsum1} \|_{L^2} \lesssim \int_0^t\|xf\|_{L^2}\|e^{is\Lambda}(n_0+G+\Lambda^{-1}n_1)\|_{L^\infty}\,\mathrm{d}s 
  \lesssim \varepsilon_1^2 \int_0^t s^\delta\frac{1}{s}\,ds \lesssim \varepsilon_1^2 t^\delta. \]

For the term \eqref{eqn: xf22}, we first split $g$ into $n_0+G$ and $i\Lambda^{-1}n_1$.
\[\|\eqref{eqn: xf22}\| \lesssim \int_0^t\|e^{is\Lambda}x(n_0+G)\|_{L^3}\|e^{is\Delta}f\|_{L^6}\,\mathrm{d}s +\int_0^t \|e^{is\Lambda}x\Lambda^{-1}n_1e^{is\Delta}f\|_{L^2} \,\mathrm{d}s.\]
For the first term, we use the Sobolev embedding and our assumption \eqref{data2}:
\begin{align*}
\int_0^t \|x(n_0+G)\|_{L^3}\|e^{is\Delta}f\|_{L^6}\,ds &\lesssim \int_0^t \|x(n_0+G)\|_{\dot{H}^{1/2}}\|e^{is\Delta}f\|_{L^6}\, \mathrm{d}s
\lesssim \varepsilon_1^2 \int_0^t s^{-1+\delta}\, \mathrm{d}s
\lesssim \varepsilon_1^2 t^\delta.
\end{align*}
For the second term, we commute $x$ and $\Lambda^{-1}$, using $x\Lambda^{-1}n_1 = -\frac{\partial}{\Lambda^3}n_1 + \Lambda^{-1}xn_1$, we get:
\begin{equation} \label{eqn: xf22a}
\int_0^t \|e^{is\Lambda}x\Lambda^{-1}n_1e^{is\Delta}f\|_{L^2} \,\mathrm{d}s \lesssim \int_0^t\|e^{is\Lambda}\Lambda^{-1}xn_1e^{is\Delta}f\|_{L^2}\,\mathrm{d}s + \int_0^t\|e^{is\Lambda}\frac{\partial}{\Lambda^3}n_1e^{is\Delta}f\|_{L^2} \,\mathrm{d}s.
\end{equation}
For the first integral on the right-hand side of \eqref{eqn: xf22a} we use H\"older and the dispersive estimates \eqref{disp} 
and \eqref{linearwave0}, to obtain
\begin{align*}
\int_1^t \|e^{is\Lambda} \Lambda^{-1} xn_1e^{is\Delta}f\|_{L^2} \,\mathrm{d}s
& \lesssim \int_1^t \|e^{is\Lambda} \Lambda^{-1} x n_1 \|_{L^3}\|e^{is\Delta}f\|_{L^6}\,\mathrm{d}s 
\lesssim \varepsilon_1 \int_1^t s^{-4/3+\delta}\|\Lambda^{-1/3}xn_1\|_{L^{3/2}}\,\mathrm{d}s
\\
&\lesssim \varepsilon_1 \int_1^t s^{-4/3+\delta}\|x n_1\|_{L^{9/7}}\,\mathrm{d}s
\lesssim \varepsilon_1 \|x^2 n_1\|_{L^2}.
\end{align*}
For the second term in \eqref{eqn: xf22a}, we have
\begin{align*}
\int_0^t\|e^{is\Lambda}\frac{\partial}{\Lambda^3}n_1e^{is\Delta}f\|_{L^2} \,\mathrm{d}s
& \lesssim \int_0^t \|e^{is\Lambda}\frac{\partial}{\Lambda^3}n_1 \|_{L^\infty}\|e^{is\Delta}f\|_{L^2} \,\mathrm{d}s
\\ & 
\lesssim \int_0^t s^{-1} \|\frac{\partial^3}{\Lambda^3}n_1\|_{L^1} \|f\|_{L^2} \,\mathrm{d}s 
\lesssim \e_1 t^\delta \|n_1\|_{\dot{B}^0_{1,1}}.
\end{align*}

To estimate the term \eqref{eqn: xf23} we first observe that \eqref{eqn: Lminus1G3},
and the assumptions \eqref{data} on the initial data combined with \eqref{linearwave}, give us
\begin{align} 
\label{estG+n_0}
\begin{split}
\|e^{is\Lambda}\Lambda^{-1}(n_0+G)\|_{L^3} 
& \lesssim \e_1 s^{-2\a-3\d} \, .
\end{split}
\end{align}
We then have
\begin{align*}
\|\eqref{eqn: xf23}\|_{L^2} &
  \lesssim t \|e^{it\Delta}f\|_{L^6} \|e^{it\Lambda}\Lambda^{-1}(n_0+G)\|_{L^3} 
  + t\|f\|_{L^2}\|e^{it\Lambda} \Lambda^{-2}m_0(i\nabla)n_1\|_{L^\infty} 
\\
& \lesssim t\, \e_1 t^{-1+\delta} \, \e_1 t^{-2\alpha+3\delta} + t \, \e_1 \frac{1}{t} \|n_1\|_{\dot{B}^0_{1,1}} \lesssim \varepsilon_1^2 t^\delta .
\end{align*}
Here, and in the remainder of the proof, we denote by $m_k$ a homogeneous multiplier of order $k$. We also implicitly use the fact that such multipliers operate on homogeneous Besov spaces $\dot{B}^{s+k}_{p,r}\rightarrow \dot{B}^s_{p,r}$. 

For the term \eqref{eqn: xf24}, we observe that $e^{is\Lambda}\partial_s g = \Lambda^\gamma |u|^2$, so that
\begin{align*}
\|\eqref{eqn: xf24}\|_{L^2} & \lesssim \int_1^t s \|e^{is\Delta}f\|_{L^6} \|\Lambda^{-1+\gamma}|u|^2\|_{L^3} \mathrm{d}s  
  \lesssim \varepsilon_1 \int_1^t s s^{-1+\delta} \|u\|^2_{L^{6/(2-\gamma)}}\mathrm{d}s  \lesssim \varepsilon_1^2 t^\delta \, ,
\end{align*}
where we use Sobolev's embedding for the second inequality.

Finally, for the term \eqref{eqn: xf25}, we use $e^{is\Delta}\partial_sf = uw$ to estimate
\begin{align*}
\|\eqref{eqn: xf25} \|_{L^2} & \lesssim \int_1^t s\|e^{is\Delta}\partial_s f\|_{L^6}\|e^{is\Lambda} \Lambda^{-1}(n_0+G)\|_{L^3} 
  + s\|\partial_s f\|_{L^2}\|e^{is\Lambda}\Lambda^{-2}m_0(i\nabla)n_1\|_{L^\infty} \, \mathrm{d}s \\
& \lesssim \int_1^t s \|u\|_{L^6}\|w\|_{L^\infty} s^{-2\alpha+3\delta} +s\|u\|_{L^\infty}\|w\|_{L^2}\frac{1}{s}\|n_1\|_{\dot{B}^0_{1,1}} \, \mathrm{d}s 
\\
& \lesssim \varepsilon_1^2 \int_1^t s s^{-1+\delta} s^{-1} s^{-2\alpha+3\delta} + s s^{-1-\alpha} s^{-1} \, \mathrm{d}s  
  \lesssim \varepsilon_1^2 t^\delta \, ,
\end{align*}
having used $\alpha > 3 \delta /2$.
\end{proof}

\subsection{Estimate for $x^2 F$}
In this last section we want to estimate ${\|x^2 F\|}_{L^2}$ and show
\begin{proposition}
Under the apriori assumptions \eqref{eqn: gL}--\eqref{apriorif10} we have
\begin{align*}
{\|x^2 F\|}_{L^2} \lesssim \e_1^2 t^{1-2\a-\d} \, .
\end{align*}
\end{proposition}

Fix $i$ and differentiate twice with respect to $\xi_i$ in \eqref{eqn: Fdef}, generating three terms:
\begin{align}
\widehat{F}_1&=-4\int_0^t\int \ephi s^2\eta_i^2 \widehat{f}(\xi-\eta)\widehat{g}(\eta,s)\dds\label{eqn: F1}\\
\widehat{F}_2&=-4\int_0^t\int \ephi is \eta_i \partial_{\xi_i} \fg\dds \label{eqn: F2} \\
\widehat{F}_3&=\int_0^t\int \ephi \partial_{\xi_i}^2\fg \dds.
\label{eqn: F3}
\end{align}

\subsubsection{Estimate for $F_1$}
As a consequence of the second equality in \eqref{eqn: null}, we have
\begin{equation}
\eta_i^2 = \frac{\eta_i^2}{|\eta|}\frac{\eta}{|\eta|}\cdot \partial_\eta \phi +\frac{\phi}{|\eta|}\frac{\eta_i^2}{|\eta|}.
\end{equation}
Plugging this  into \eqref{eqn: F1} we obtain two terms (we omit the constant factor):
\begin{align}
& \int_0^t\int s^2 \ephi \frac{\eta_i^2}{|\eta|}\dirphi \, \fg\dds \label{eqn: F1-dphiterm} \, ,
\\
& \int_0^t \int s^2 \ephi\frac{\eta_i^2}{|\eta|}\phiovereta \fg \dds \label{eqn: F1-phiterm} \, .
\end{align}

\nl
{\it Integration by parts in $s$ -- the term \eqref{eqn: F1-phiterm}}. Note that
\[\phi \ephi =-i \partial_s \ephi, \]
and so \eqref{eqn: F1-phiterm} may be rewritten as a sum of the terms
\begin{align}
\label{eqn: F11}
\widehat{F}_{11}&=-\int i m_0(\eta)t^2 \ephi \widehat{f}(\xi-\eta,t)\widehat{g}(\eta,t) \mathrm{d}\eta
\\ 
& \quad + i\int_0^t\int m_0(\eta) s^2 \, \ephi \partial_s \left( \fg \right)\dds, \nonumber 
\\
\label{eqn: F13}
\widehat{F}_{13}&=2i\int_0^t\int m_0(\eta) s \, \ephi \fg \dds \, . 
\end{align}

\nl
{\it Integration by parts in $\eta$ -- the term \eqref{eqn: F1-dphiterm}}. 
Write out \eqref{eqn: F1-dphiterm} as a sum of terms
\[\int_0^t\int s^2 \frac{\eta_i^2}{|\eta|^2} \eta_j \partial_{\eta_j}\phi \, \ephi \fg\dds.\]
Fix one summand. Since
$s\partial_{\eta_j}\phi e^{is\phi} =\partial_{\eta_j}e^{is\phi}$,
integration by parts yields
\begin{align}
& -\int_0^t \int s \partial_{\eta_j}\left(\frac{\eta_i^2\eta_j}{|\eta|^2}\right)\ephi \fg \,\mathrm{d}\eta\mathrm{d}s \label{eqn: dhitssymbol1} 
\\
& -\int_0^t \int s \frac{\eta_i^2\eta_j}{|\eta|^2} \ephi \partial_{\eta_j} \left(\fg \right) \,\mathrm{d}\eta\mathrm{d}s \, .\label{eqn: dfg}
\end{align}
The first term is analogous to $\hat{F}_{13}$ in \eqref{eqn: F13}.
The quantity \eqref{eqn: dfg} gives two contributions:
\begin{align}
& \int_0^t \int m_1(\eta) s  \ephi \partial_{\eta_j}\widehat{f}(\xi-\eta)\widehat{g}(\eta)\dds
\\
& \int_0^t\int m_1(\eta) s  \ephi \widehat{f}(\xi-\eta)\partial_{\eta_j}\widehat{g}(\eta)\dds \, , \label{eqn: F12} 
\end{align}
where the symbol $m_1(\eta)$ denotes a multiplier which is homogenous of order 1. 
The first term above is analogous to the term $F_2$ in \eqref{eqn: F2} and will be estimated later. 
We denote the second term by $\widehat{F}_{12}$. 
We now proceed to estimate the terms $F_{11}$, $F_{12}$ and $F_{13}$,
defined respectively in \eqref{eqn: F11}, \eqref{eqn: F12} and \eqref{eqn: F13}.

\nl
{\it Estimate for $F_{11}$}. This term was defined in \eqref{eqn: F11}. The first term
\[\int m_0(\eta) t^2 \ephi \widehat{f}(\xi-\eta,t)\widehat{g}(\eta,t) \mathrm{d}\eta\]
can be dealt with by an $L^6-L^3$ estimate: we pair the zero-th order multiplier with $g=n_0+G +\Lambda^{-1}n_1$ to find a bound of
\begin{align*}
& t^2 {\|e^{it\Delta} f\|}_{L^6}  {\|e^{is\Lambda}(n_0+G)\|}_{L^3} + t^2\|xf\|_{L^2}\|e^{it\Lambda}\Lambda^{-1}n_1\|_\infty 
\\
& \lesssim \varepsilon_1 t {\|xf\|}_{L^2} t^{-1/3} + \varepsilon_1^2 t^\delta 
\lesssim \varepsilon_1^2 t^{2/3+\delta} .
\end{align*}
This is acceptable since $2/3 + \d \leq 1 - 2\alpha -\d$, see \eqref{eqn: alpha}.
For the second term in \eqref{eqn: F11}, we first expand the derivative
\begin{align*}
\partial_s \left(\widehat{f}(\xi-\eta)\widehat{g}(\eta)\right) = \partial_s\widehat{f}(\xi-\eta)\widehat{g}(\eta)
  + \widehat{f}(\xi-\eta)\partial_s \widehat{g}(\eta).
\end{align*}
Using
$\partial_s f = e^{is\Delta} un$ and $\partial_s g = e^{is\Lambda} \Lambda^{\g} |u|^2$,
we obtain a bound of the form
\[ \int_0^t s^2 {\|un\|}_{L^6} {\|n\|}_{L^3}\,\mathrm{d}s 
  + \int_0^t s^2 s^\delta {\|u\|}_{L^6} {\||u|^2\|}_{L^3}\,\mathrm{d}s \, .\]
Using the a priori assumptions \eqref{eqn: gL}, this is bounded by
\begin{align*}
& \varepsilon_1^2 \int_0^t s s^{-1+\delta} s^{-1/3} \,\mathrm{d}s + \varepsilon_1^2 \int_0^t s^{2+\delta} {(s^{-1+\d})}^3 \,\mathrm{d}s
\lesssim \varepsilon_1^2 \int_0^t s^{\delta -1/3}\,\mathrm{d}s \lesssim \varepsilon_1^2 t^{2/3 + \delta} \, .
\end{align*}

\nl
{\it Estimate for $F_{12}$.} The term $F_{12}$ was defined  in \eqref{eqn: F12} as a term of the form
\[\widehat{F}_{12} = \int_0^t \int m_1(\eta) s \, \ephi \widehat{f}(\xi-\eta)\partial_{\eta}\widehat{g}(\eta)\dds.\]
Up to a commutator term resulting from
$\eta \partial_\eta \widehat{g}(\eta) = \partial_{\eta}(\eta \widehat{g}(\eta)) -\widehat{g}(\eta)$,
we can apply H\"older's inequality with exponents
$p = 4/(1+\gamma)$ and  $p'= 4/(1-\gamma)$
to find a bound of the type
\[\int_0^t s \| m_0(\nabla) e^{is\Lambda} x \Lambda g \|_{L^{4/(1+\gamma)}}\|e^{is\Delta}f\|_{L^{4/(1-\gamma)}} \,\mathrm{d}s.\]
By \eqref{eqn: gL},
\[\| m_0(\nabla) e^{is\Lambda} x \Lambda (n_0+G) \|_{L^{4/(1+\gamma)}} \lesssim \varepsilon_1 s^{-1/4+3\gamma/4-2\alpha-3\delta} \, \]
and, by interpolating the $L^2$ and $L^6$ bounds, we have
\[\|e^{is\Delta} f\|_{L^{4/(1-\gamma)}}\lesssim \varepsilon_1 s^{-(3/4+3\gamma/4)}.\]

On the other hand, we have
\begin{align*}
\| e^{is\Lambda} x n_1\|_{L^{4/(1+\gamma)}} &\lesssim t^{-1/2+\gamma/2}\|\Lambda^{1-\gamma}(xn_1)\|_{L^{4/(3-\gamma)}}\\
&\lesssim t^{-1/2+\gamma/2}(\|x\Lambda n_1\|_{L^{12/(9+\gamma)}} + \|n_1\|_{L^{12/(9+\gamma)}})\\
&\lesssim t^{-1/2+\gamma/2}\|\langle x\rangle^2 \Lambda n_1\|_{L^2}.
\end{align*}
Combining these, we obtain the result.

\nl
{\it The term $F_{13}$.} This was defined in \eqref{eqn: F13}. Taking $L^6-L^3$ we have a bound of the form
\[
\int_0^t s \|u\|_{L^6}\|n\|_{L^3}\,\mathrm{d}s +\int_0^t s\|u\|_{L^6}\|e^{is\Lambda}\Lambda^{-1}n_1\|_{L^3}\,\mathrm{d}s.
\]
For the first term in the preceding, we have the bound 
\[\int_0^t s^{\delta-1/3}\,\mathrm{d}s  \lesssim t^{2/3+\delta}.\]
For the second term, we use the dispersive estimate and Sobolev embedding to obtain the bound
$t^{2/3}\|n_1\|_{L^{9/7}}$,
which is more than what is needed.

\subsubsection{Estimate for $F_2$} This was defined in \eqref{eqn: F2} and is of the form
\[\int_0^t\int \ephi i s \eta_i \partial_{\xi_i} \fg\dds.\]
We use \eqref{eqn: null} to obtain the two terms
\begin{align}
& -\int_0^t\int \ephi is \frac{\eta_i}{|\eta|} \dirphi \, \partial_{\eta_i} \fg\dds \label{eqn: F2-dphiterm}
\\
& -\int_0^t\int \ephi is \frac{\eta_i}{|\eta|}\frac{\phi(\xi,\eta)}{|\eta|} \partial_{\eta_i} \fg\dds \label{eqn: F2-phiterm} \, .
\end{align}
We integrate by parts in \eqref{eqn: F2-dphiterm}, obtaining terms
\begin{align}
& \int_0^t \partial_{\eta_j}\left(\frac{\eta_i\eta_j}{|\eta|^2}\right) \ephi \partial_{\eta_i} \fg\dds \label{eqn: F2-term1}
\\
& \int_0^t \frac{\eta_i\eta_j}{|\eta|^2}\ephi \partial^2_{\eta_i\eta_j}\fg\dds \label{eqn: F2-term2}
\\
&\int_0^t \frac{\eta_i\eta_j}{|\eta|^2}\ephi \partial_{\eta_i}\widehat{f}(\xi-\eta)\partial_{\eta_j}\widehat{g}(\eta,s)\dds \, . \label{eqn: F2-term3}
\end{align}
Notice that the multiplier in \eqref{eqn: F2-term1} is of order $-1$ while in \eqref{eqn: F2-term2} and \eqref{eqn: F2-term3} it is of order zero. 
In particular, \eqref{eqn: F2-term2} is similar to \eqref{eqn: F3} and can be estimated as in \eqref{estF_3} below, so we skip it.

For \eqref{eqn: F2-term1}, we use the Mikhlin multiplier Theorem\footnote{$|\eta| m_{-1}(\eta)$ 
is bounded on $L^p$, $1<p<\infty$.} and the second assumption in \eqref{eqn: Lminus1G3} and \eqref{apriorif10}, 
to find the bound
\begin{align*}
{\| \eqref{eqn: F2-term1} \|}_{L^2} & \lesssim
  \int_0^t {\|e^{is\Delta}xf\|}_{L^6} {\|e^{is\Lambda} m_{-1}(\nabla)(n_0+G) \|}_{L^3} \,\mathrm{d}s
  +\int_0^t \|xf\|_{L^2}\|e^{is\Lambda}\Lambda^{-2}n_1\|_{L^\infty}\,\mathrm{d}s\\
&\lesssim \int_0^t s^{-1} {\|x^2 f\|}_{L^2} {\|e^{is\Lambda} \Lambda^{-1} (n_0+G)\|}_{L^3} \,\mathrm{d}s + \varepsilon_1^2 t^\delta\\
& \lesssim \varepsilon_1^2 \int_0^t s^{-1} s^{1-2\a-\d} s^{-2\alpha-3\delta}\,\mathrm{d}s + \varepsilon_1^2 t^\delta
  \lesssim \varepsilon_1^2 t^{1-2\alpha-3\delta}\, .
\end{align*}

The term \eqref{eqn: F2-term3} can be estimated similarly using Sobolev's embedding and \eqref{eqn: gL}:
\begin{align*}
{\| \eqref{eqn: F2-term3} \|}_{L^2} & \lesssim
  \int_0^t {\|e^{is\Delta}xf\|}_{L^6} {\| e^{is\Lambda} x(n_0+G) \|}_{L^3} \,\mathrm{d}s + \int_0^t \|e^{is\Delta}xf\|_{L^6}
  \|e^{is\Lambda}\Lambda^{-1}xn_1\|_{L^3}\,\mathrm{d}s\\
& + \int_0^t \|e^{is\Delta}xf\|_{L^2}\|e^{is\Lambda}\frac{\partial}{\Lambda^3}n_1\|_{L^\infty}\,\mathrm{d}s \\
& \lesssim \int_0^t s^{-1} {\|x^2 f\|}_{L^2} {\| x(n_0+G) \|}_{\dot{H}^\frac{1}{2}} \,\mathrm{d}s +\int_0^t s^{-1} \|x^2 f\|_{L^2}\|\Lambda^{-1/3} x n_1\|_{L^{3/2}}\,\mathrm{d}s  \\
& + \int_0^t s^{-1}\|\frac{\partial^3}{\Lambda^3}n_1\|_{L^1} \|xf\|_{L^2}\,\mathrm{d}s \\
& \lesssim \varepsilon_1^2 \int_0^t s^{-1} s^{1-2\a-\d} \,\mathrm{d}s + \varepsilon_1 \int_0^t s^{-2\alpha-\delta}s^{-1/3}\|xn_1\|_{L^{9/7}}\,\mathrm{d}s 
+ \varepsilon_1^2 t^\delta\\ 
&\lesssim \varepsilon_1^2 t^{1-2\alpha-\delta}.
\end{align*}

We now integrate by parts in $s$ in the term \eqref{eqn: F2-phiterm} to find the terms
\begin{align}
&-\int_0^t\int \ephi s \frac{\eta_i}{|\eta|^2} \partial_s \partial_{\eta_i} \widehat{f}(\xi-\eta,s)\widehat{g}(\eta,s)\,\mathrm{d}\eta \mathrm{d}s 
&\quad (\equiv \widehat{F}_{21}) \label{eqn: F21} 
\\
& -\int_0^t\int \ephi s \frac{\eta_i}{|\eta|^2} \partial_{\eta_i} \widehat{f}(\xi-\eta,s)\partial_s \widehat{g}(\eta,s)\,\mathrm{d}\eta \mathrm{d}s 
&\quad (\equiv \widehat{F}_{22}) \label{eqn: F22}
\\
& \int e^{it\phi(\xi,\eta)} t \frac{\eta_i}{|\eta|^2} \partial_{\eta_i} \widehat{f}(\xi-\eta,t)\widehat{g}(\eta,t)\mathrm{d}\eta
& \quad (\equiv \widehat{F}_{23}) \label{eqn: F23}
\\
& -\int_0^t\int \ephi \frac{\eta_i}{|\eta|^2}  \partial_{\eta_i} \fg\dds \, . &\quad (\equiv \widehat{F}_{24}) \label{eqn: F24}
\end{align}
We estimate these four terms below.

\nl
{\it Estimate for $F_{21}$.} We first compute the $\partial_s$ derivative:
\[\partial_s \widehat{f}(\xi-\eta,s) = i\int e^{is\phi(\xi-\eta,\tau)} \widehat{f}(\xi-\eta-\tau,s)\widehat{g}(\tau,s)\,\mathrm{d}\tau.\]
The derivative $\partial_{\eta_j}$ now generates two terms:
\begin{gather}
2i\int_0^t\int\int \ephi s^2 \frac{\eta_i}{|\eta|^2} e^{is\phi(\xi-\eta,\tau)}
  \widehat{f}(\xi-\eta-\tau,s)\tau_j \widehat{g}(\tau,s)\widehat{g}(\eta,s)\,\mathrm{d}\eta\mathrm{d}\tau\mathrm{d}s  \label{eqn: F21-term1} 
\\
\int_0^t\int\int \ephi s \frac{\eta_i}{|\eta|^2} e^{is\phi(\xi-\eta,\tau)}
  \partial_{\eta_j}\widehat{f}(\xi-\eta-\tau,s)\widehat{g}(\tau,s) \widehat{g}(\eta,s)\,\mathrm{d}\eta\mathrm{d}\tau\mathrm{d}s. \label{eqn: F21-term2}
\end{gather}
For the first term we pair the multiplier with $\widehat{g}(\eta)$ and use $L^6-L^3$ to obtain
\begin{equation}
{\| \eqref{eqn: F21-term1} \|}_{L^2} \lesssim \int_0^t s^2 {\|u\partial n\|}_{L^6} {\|e^{is\Lambda} \Lambda^{-1} (G+n_0)\|}_{L^3}
  \,\mathrm{d}s+\int_0^t s^2\|u\partial n\|_{L^2}\|e^{is\Lambda} \Lambda^{-2}n_1\|_{L^\infty}\,\mathrm{d}s .\label{eqn: F21-term3}
\end{equation}
Estimating $\|e^{is\Lambda} \Lambda^{-1} (G + n_0)\|_{L^3}$ by \eqref{estG+n_0},
we can bound the first of the two expressions above by
\begin{align*}
\int_0^t s^2 {\|u\|}_{L^6} {\|\partial n\|}_{L^\infty} {\|e^{is\Lambda} \Lambda^{-1} (G+n_0)\|}_{L^3}\mathrm{d}s 
  &\lesssim \varepsilon_1^2 \int_0^t s^{2\d} s^{-2\alpha -3\d } \mathrm{d}s \lesssim  \varepsilon_1^2 t^{1-2\alpha-\delta}.
\end{align*}
For the second integral in \eqref{eqn: F21-term3} we have
\begin{align*}
\int_0^t s^2\|u\partial n\|_{L^2}\|e^{is\Lambda} \Lambda^{-2}n_1\|_{L^\infty}\,\mathrm{d}s &\lesssim \int_0^t s\|u\|_{L^6}
  \|\partial n\|_{L^3}\|n_1\|_{\dot{B}^0_{1,1}}\,\mathrm{d}s \\
&\lesssim \int_0^t \varepsilon_1^3 s^{-1/3+2\delta}\,\mathrm{d}s
\lesssim \varepsilon_1^3 t^{2/3+2\delta}.
\end{align*}

Moving on to \eqref{eqn: F21-term2}, we reproduce the same pairing as before and once again use $L^6-L^3$ and $L^2-L^\infty$ estimates
to find a bound
\begin{equation}\label{eqn: F21-term4}
\int_0^t s\|e^{is\Delta}(xf) \,n \|_{L^6} \|e^{is\Lambda}\Lambda^{-1} (n_0+G)\|_{L^3}\mathrm{d}s 
  + \int_0^t s \| e^{is\Delta}(xf) \, n \|_{L^2}\|e^{is\Lambda}\Lambda^{-2}n_1\|_{L^\infty}\,\mathrm{d}s.
\end{equation}
Observe that 
\[ \| e^{is\Delta}xf \|_{L^6} \lesssim s^{-1} \| x^2 f \|_{L^2} \lesssim \e_1 s^{-2\alpha - \d} \, .\]
Using this, together with $\| n\|_{L^\infty} \le \varepsilon_1 s^{-1}$ and \eqref{eqn: Lminus1G3},
 we obtain the desired bound of $\varepsilon_1^2 t^{1-2\a-\d}$ for the first term in \eqref{eqn: F21-term4}. 
For the second term, we can obtain a bound of $\e_1^3 t^\d$ by using
\[ \| e^{is\Delta}(xf) \, n \|_{L^2} \le \|xf\|_{L^2}\|n\|_{L^\infty} \lesssim \e^2_1 s^{-1 + \delta} \, ,\]
combined with the the linear dipersive estimate \eqref{linearwave} and the assumption on $n_1$ in \eqref{data2}.

\nl
{\it Estimate for $F_{22}$.} This term is defined in \eqref{eqn: F22}. 
We pair the multiplier with $\partial_s g = e^{is\Lambda} \Lambda^{\g} |u|^2$, use an $L^6-L^3$ estimate,
Sobolev's embedding and \eqref{apriorif10}, to find the bound
\begin{align*}
 {\| \eqref{eqn: F22} \|}_{L^2} & \lesssim \int_0^t s\|e^{is\Delta}xf\|_{L^6} {\|\Lambda^{-1+\g} {|u|}^2 \|}_{L^3} \,\mathrm{d}s
  \lesssim \varepsilon_1 \int_0^t s^{1-2\a-\d} {\| {|u|}^2 \|}_{L^{3/(2+\g)}} \,\mathrm{d}s 
\\
& \lesssim \varepsilon_1^2 \int_0^t s^{1-2\a-\d} {\| u \|}_{L^2} {\| u \|}_{L^{6/(1-2\g)}} \,\mathrm{d}s \lesssim \varepsilon_1^2 t^{1-2\a-\d} \, .
\end{align*}

\nl
{\it Estimate for $F_{23}$.} This term is defined in \eqref{eqn: F23}.
 We can use $L^6-L^3$ and $L^\infty-L^2$ as before to get
\begin{align*} {\| \eqref{eqn: F23} \|}_{L^2} &\lesssim t\|e^{it\Delta} xf\|_{L^6}\|\Lambda^{-1}e^{is\Lambda}(n_0+G)\|_{L^3}
\\
& +t\|xf\|_{L^2}\|e^{it\Lambda}\Lambda^{-2}n_1\|_{L^\infty}
\lesssim \varepsilon_1^2 t^{1-4\alpha-5\delta} \, 
\end{align*}

\nl
{\it Estimate for $F_{24}$.} This was defined in \eqref{eqn: F24}. It is dealt with by $L^6-L^3$ and $L^\infty-L^2$, in an identical fashion to the previous case.

\subsubsection{Estimate for $F_3$}
This was defined in \eqref{eqn: F3}. It can be dealt with by using an $L^2-L^\infty$ estimate leading to
\begin{align}
\label{estF_3}
\begin{split}
{\| (\eqref{eqn: F3} \|}_{L^2} &\lesssim \int_0^t {\|x^2f\|}_{L^2} {\|e^{is\Lambda}(n_0+G)+e^{is\Lambda}\Lambda^{-1}n_1\|}_{L^\infty}\,\mathrm{d}s\\
&\lesssim \int_0^t \|x^2f\|_{L^2}\|n\|_{L^\infty}\,\mathrm{d}s +\int_0^t \|x^2f\|_{L^2}\|\Lambda n_1\|_{\dot{B}^0_{1,1}}\,\mathrm{d}s
\lesssim \varepsilon_1^2 t^{1-2\alpha-\delta} \, . 
\end{split}
\end{align}
At this point all terms are accounted for. $\hfill \Box$

\end{document}